\documentclass[11pt]{amsart}

\usepackage[T1]{fontenc}
\usepackage{lmodern}
\usepackage{amsmath,amssymb,amsthm,mathtools}
\usepackage{enumitem}
\usepackage{microtype}
\usepackage{hyperref}
\usepackage[nameinlink,capitalize]{cleveref}
\usepackage{mathrsfs}
\usepackage{xcolor}
\usepackage{tikz}
\usetikzlibrary{positioning,calc}
\crefname{theorem}{Theorem}{Theorems}
\crefname{proposition}{Proposition}{Propositions}
\crefname{lemma}{Lemma}{Lemmas}
\crefname{corollary}{Corollary}{Corollaries}
\crefname{remark}{Remark}{Remarks}
\crefname{definition}{Definition}{Definitions}
\crefname{question}{Question}{Questions}
\crefname{example}{Example}{Examples}

\hypersetup{
  colorlinks=true,
  linkcolor=blue,
  citecolor=blue,
  urlcolor=blue
}

\newtheorem{theorem}{Theorem}[section]
\newtheorem{proposition}[theorem]{Proposition}
\newtheorem{lemma}[theorem]{Lemma}
\newtheorem{corollary}[theorem]{Corollary}
\newtheorem{remark}[theorem]{Remark}
\newtheorem{question}[theorem]{Question}

\theoremstyle{definition}
\newtheorem{definition}[theorem]{Definition}

\newtheorem{example}[theorem]{Example}

\newcommand{\OO}{\mathcal O}

\newcommand{\Pic}{\operatorname{Pic}}
\newcommand{\Alb}{\operatorname{Alb}}
\newcommand{\rk}{\operatorname{rk}}
\newcommand{\NS}{\operatorname{NS}}
\newcommand{\cris}{\mathrm{cris}}
\newcommand{\et}{\mathrm{\acute et}}

\newcommand{\Jac}{\operatorname{Jac}}

\newcommand{\PP}{\mathbf P}
\newcommand{\FF}{\mathbf F}

\title[Arakelov Inequalities in Positive Characteristic]
{Arakelov inequalities for fibered surfaces in positive characteristic}

\author{Hao Max Sun}
\address{Department of Mathematics, Shanghai Normal University, Shanghai, 200234, P.\ R.\ China}
\email{hsun@shnu.edu.cn}

\author{Wan-Yuan Xu}
\address{Department of Mathematics, Shanghai Normal University, Shanghai, 200234, P.\ R.\ China}
\email{wanyuanxu@shnu.edu.cn}


\begin{document}

\begin{abstract}
Let $f:S\to C$ be a relatively minimal semistable fibration of genus $g\ge2$ over an algebraically closed field of characteristic $p>0$ with smooth and geometrically connected generic fiber. Let $\Sigma\subset C$ denote the set of points over which the fibers are singular, with $s=|\Sigma|$. If the total surface $S$ is Hodge Witt and $s\ge2$, we prove the classical Arakelov inequality
$$
 \deg f_*\omega_{S/C}\le \frac g2\deg\Omega_C^1(\log\Sigma),
$$
where $\omega_{S/C}$ denotes the relative dualizing sheaf. For $C=\mathbb P^1$, we obtain the stronger estimate
$$
 \deg f_*\omega_{S/\mathbb P^1}\le \frac g2(s-2)-\frac12 b_1(S),
$$
where $b_1(S)$ is the first Betti number of $S$. We also construct a genus-$2$ semistable fibration over $\mathbb P^1$ with exactly $4$ singular fibers in characteristic $5$. Its total surface is Hodge Witt, and
$$
 \deg f_*\omega_{S/\mathbb P^1}=2=\frac g2(s-2).
$$
Thus Nguyen's lower bound $s\ge4$ and the Hodge--Witt Arakelov inequality are both sharp. As applications, we derive a canonical-class inequality and a Szpiro-type inequality with coefficients linear in $g$.
\end{abstract}

\maketitle

\section{Introduction}

Let $f:S\to C$ be a relatively minimal semistable fibration of genus $g\ge2$ from a smooth projective surface to a smooth projective curve over an algebraically closed field of characteristic $p>0$. We assume that the geometric generic fiber is smooth and geometrically connected. Denote the genus of $C$ by $b$. Let $\Sigma\subset C$ be the set of the image of singular fibers, and put $s=|\Sigma|$. The classical Arakelov inequality is
\begin{equation}\label{eq:classical-arakelov-intro}
 \deg f_*\omega_{S/C}
 \le \frac g2\deg\Omega_C^1(\log\Sigma)
 =\frac g2(2b-2+s).
\end{equation}
In characteristic zero, Arakelov inequalities are a key boundedness input in proofs of the Shafarevich finiteness theorem: For fixed $C$ and $\Sigma$, there are only finitely many isomorphism classes of non-isotrivial families of smooth projective curves over $C-\Sigma$.

Arakelov first proved this type of inequality for families of curves over a complex base \cite{Arakelov1971}. Faltings proved the corresponding result for families of abelian varieties \cite{Faltings1983}. Later, Jost--Zuo and Viehweg--Zuo developed a Higgs-bundle approach to Arakelov inequalities \cite{JostZuo2002,ViehwegZuo2006}. Lu--Tan--Yu--Zuo gave another proof for fibered surfaces over $\mathbb C$. Their proof is based on a lower bound for $h^{1,1}(S)$ \cite{LTYZ2014}. Arakelov inequalities also play an important role in the study of Shimura subvarieties in the Torelli locus; see the works of Chen--Lu--Zuo and
Lu--Zuo \cite{ChenLuZuo2016,LuZuo2017,LuZuo2019}.

The situation in positive characteristic is different. Lu-Sheng-Zuo proved an Arakelov-type inequality with coefficient $2g^2$ under a nonzero Kodaira--Spencer assumption \cite{LuShengZuo2009}. Yuan-Zhang obtained a general bound with coefficient $g^2$ \cite{YuanZhang2014}. Under a lifting hypothesis, Xu recovered the classical coefficient $\frac{g}{2}$ \cite{Xu2018}. 

This raises a natural question as follows. \begin{question}Can the coefficient $\frac{g}{2}$ be recovered from an intrinsic condition on the total surface, without the strong lifting assumption? \end{question}

Our first main result gives such a criterion in terms of Hodge Witt cohomology.

\begin{theorem}\label{thm:hodge-witt}
Let $f:S\to C$ be a relatively minimal semistable fibration whose generic fiber is smooth and geometrically connected of genus $g\ge2$. Assume that $S$ is Hodge Witt and that $s\ge2$. Then
\begin{equation}\label{eq:HW-arakelov}
 \deg f_*\omega_{S/C}\le \frac g2(2b-2+s).
\end{equation}
\end{theorem}

\begin{remark}
The Hodge Witt hypothesis is mild. Firstly, it is strictly weaker than ordinarity. Secondly, it is automatic when $p_g(S)=0$. Moreover, by a theorem of Crew, it is also an open condition in smooth families of surfaces. See Subsection~\ref{subsec:HW-vs-ordinary} for details.  
\end{remark}

The proof of Theorem~\ref{thm:hodge-witt} is based on a more general estimate. For an arbitrary relatively minimal semistable fibration as above, we prove
\begin{equation}\label{eq:refined-intro}
 \deg f_*\omega_{S/C}
 \le \frac g2(2b-2+s)
 +T^{0,2}(S)+\frac12\bigl(r_f-G_f\bigr).
\end{equation}
Here $T^{0,2}(S)$ is the domino number, while $r_f-G_f$ is controlled by the relative Albanese geometry of the singular fibers. Thus the two terms in \eqref{eq:refined-intro} record the positive-characteristic corrections that occur in our method. The Hodge Witt condition gives $T^{0,2}(S)=0$, and for $s\ge2$ the Albanese term is non-positive. Hence Theorem~\ref{thm:hodge-witt} follows from this refined estimate. Its proof combines a crystalline lower bound for the slope number $m^{1,1}(S)$ with divisor classes supported on the singular fibers. 
The argument is inspired by the numerical method of Lu--Tan--Yu--Zuo \cite{LTYZ2014}, but the correction terms in \eqref{eq:refined-intro} are specific to positive characteristic.

The number of singular fibers of a non-isotrivial fibration over $\mathbb P^1$ is a classical problem. It is closely related to Arakelov and canonical-class inequalities. Over $\mathbb C$, Beauville proved that a non-isotrivial semistable fibration of genus $g\geq 1$ has at least four singular fibers and conjectured that $s\geq 5$ when $g\ge2$ \cite{Beauville1981}. Tan  \cite{Tan1995} and Liu \cite{Liu1996} independently proved Beauville's conjecture by the strict canonical-class inequality. Tan also constructed a genus-$2$ semistable fibration with five singular fibers, so the bound is sharp. Tan--Tu--Zamora proved that $s\ge6$ when $\kappa(S)\ge0$ \cite{TanTuZamora2005}. They also formulated the conjectural lower bound $s\ge7$ when the total surface is of general type. In \cite{LuTanZuo2018}, Lu--Tan--Zuo showed that $s\ge7$ when $S$ is of general type. Their theorem applies more generally to semistable families of higher-dimensional varieties. Let $s_{\mathrm{nc}}$ be the number of singular fibers with non-compact Jacobian. Lu--Tan--Xu--Zuo proved $s_{\mathrm{nc}}\ge5$ when $g\ge5$ \cite{LTXZ2016}.

In positive characteristic, the picture is subtler. Szpiro obtained $s\ge4$ when $S$ is of general type \cite{Szpiro1981}. Nguyen proved $s\ge4$ and described strong numerical restrictions in the equality case \cite{Nguyen1998}. Lu--Tan--Xu--Zuo proved $s_{\mathrm{nc}}\ge4$ when the relative Jacobian is non-smooth \cite{LTXZ2016}. Under a lifting hypothesis to Witt vectors, Xu recovered the characteristic-zero bound $s\ge5$ \cite{Xu2018}. More recently, Gong--Gu studied optimal lower bounds for surface fibrations over $\mathbb P^1$, with explicit attention to the characteristic \cite{GongGu2020}. Thus the gap between the characteristic-zero lower bound $5$ and Nguyen's positive-characteristic lower bound $4$ is a genuine issue. It does not follow formally from the existing characteristic-zero theory.

For the base $\mathbb P^1$, Nguyen's theorem makes the required lower bound on $s$ automatic. Then we have the following stronger Arakelov inequality.
\begin{theorem}\label{thm:P1}
Let $f:S\to\mathbb P^1$ be a relatively minimal non-isotrivial semistable fibration whose generic fiber is smooth and geometrically connected of genus $g\ge2$. Assume that $S$ is Hodge Witt. Then
\begin{equation}\label{eq:P1-strong}
 \deg f_*\omega_{S/\mathbb P^1}
 \le \frac g2(s-2)-\frac12 b_1(S).
\end{equation}
In particular,
\begin{equation}\label{eq:P1-classical}
 \deg f_*\omega_{S/\mathbb P^1}\le \frac g2(s-2).
\end{equation}
\end{theorem}

The remaining question is whether Nguyen's bound $s\ge4$ is attained. To the best of our knowledge, the existence of a non-isotrivial semistable family of curves of genus $g\ge2$ over $\mathbb P^1$ with $4$ singular fibers has remained open; see \cite[Remark~1.4]{Xu2018}.  

Our third main result gives such a surprising example. We reduce modulo $5$ the genus-$2$ pencil reproduced by Nguyen in \cite[\S3.2.3]{Nguyen1995}, which originates in work of Saito \cite{Saito1995}. 

\begin{theorem}\label{thm:four-fibers}
Let $k$ be an algebraically closed field of characteristic $5$.  There exists a relatively minimal, non-isotrivial, semistable fibration
\[
 f:S\longrightarrow \mathbb P^1_k
\]
of genus $2$ with exactly four singular fibers, over\quad
$
 0,\quad 1,\quad 3,\quad \infty.
$

Moreover $S$ is rational and
\[
 \deg f_*\omega_{S/\mathbb P^1}=2,\qquad
 K_{S/\mathbb P^1}^2=4,\qquad
 \delta_f=20.
\]
In particular, $S$ is Hodge Witt and
\[
 \deg f_*\omega_{S/\mathbb P^1}
 =\frac g2(s-2).
\]
\end{theorem}

Theorem~\ref{thm:four-fibers} shows that Nguyen's lower bound $s\ge4$ is already sharp in genus $2$. It also shows that the Hodge Witt Arakelov inequality is sharp in positive characteristic.

As applications of Theorems~\ref{thm:hodge-witt} and~\ref{thm:P1}, we derive a Szpiro-type inequality and a canonical-class upper bound in Section~\ref{sec:applications}. In positive characteristic, Szpiro proved the following canonical-class estimate when the Kodaira--Spencer class is nonzero:
\[
 K_{S/C}^2<4g(g-1)(2b-2+s),
\]
see \cite[Th\'eor\`eme~3]{Szpiro1981}.  Its coefficient is quadratic in $g$. Under the Hodge Witt hypothesis, our argument instead gives
\[
 K_{S/C}^2\le 6g(2b-2+s).
\]
The coefficient is linear in $g$. In particular, for $g\ge3$, the leading coefficient $6g$ is smaller than Szpiro's $4g(g-1)$.  

The paper is organized as follows. Section~\ref{sec:preliminaries} recalls the Hodge Witt and fibration-theoretic preliminaries. Sections~\ref{sec:lemmas} and~\ref{sec:proofs} prove Theorems~\ref{thm:hodge-witt} and~\ref{thm:P1}. Section~\ref{sec:four-fibers} proves Theorem~\ref{thm:four-fibers}. Section~\ref{sec:applications} gives applications, and Section~\ref{sec:further} contains further discussion.

\section{Notation and preliminaries}\label{sec:preliminaries}

All varieties are defined over an algebraically closed field $k$ of characteristic $p>0$. All curves and surfaces are projective unless stated otherwise. We now collect the notation and standard facts used in the proofs.

\subsection{General conventions}
For a smooth projective surface $S$, the first $\ell$-adic Betti number satisfies
$$b_1(S)=\dim_{\mathbb Q_\ell}H^1_{\et}(S,\mathbb Q_\ell)=2\dim\Alb(S).$$

In positive characteristic, \(\Pic^0_{S/k}\) need not be reduced. Thus \(\dim\Alb(S)\) need not equal \(h^1(S,\mathcal O_S)\).

We write $W=W(k)$ for the ring of Witt vectors and $K=\operatorname{Frac}(W)$. A prime $\ell$ is always assumed to be different from $p$. For a smooth projective variety $X$, we write
$$H^i_{\cris}(X/K):=H^i_{\cris}(X/W)\otimes_WK.$$
We will repeatedly use the following standard facts.
\begin{enumerate}[label=\textup{(P\arabic*)}]
 \item For a divisor $D$ on a smooth projective variety, functoriality of crystalline Chern classes gives
 \begin{equation}\label{eq:crys-c1-frob}
  F_{\cris}\bigl(c_1^{\cris}(D)\bigr)
   =c_1^{\cris}(F_S^*D)
   =p\,c_1^{\cris}(D).
 \end{equation}
 Hence every nonzero crystalline divisor class is a slope-$1$ class in $H^2_{\cris}(S/K)$.
 \item Crystalline cup products recover intersection numbers: for divisors $D,E$ on a smooth projective surface,
 \begin{equation}\label{eq:crys-intersection}
  \operatorname{Tr}_S\bigl(c_1^{\cris}(D)\cup c_1^{\cris}(E)\bigr)=D\cdot E.
 \end{equation}
 We use this only to test linear independence in an explicitly constructed finite-dimensional space of divisor classes. No form of the Tate conjecture is used.
 \item The intersection matrix of the irreducible components of a fiber of a fibration is negative semidefinite, with radical generated by the total fiber class.  This is the usual Zariski lemma for fiber components.
 \item The Hodge index theorem for a smooth projective surface is valid in arbitrary characteristic.  In particular, if $D^2>0$ and $0\ne Z\in\NS(S)_{\mathbb R}$ satisfies $D\cdot Z=0$, then $Z^2<0$.
\end{enumerate}

\subsection{Semistable fibrations}

By a \emph{surface fibration} we mean a proper flat surjective morphism $f: S\to C$ from a smooth projective surface to a smooth projective curve. We require the fibers to be connected and the generic fiber to be smooth and geometrically connected. We write $b=g(C)$ and denote the genus of the geometric generic fiber by $g\ge2$. Since the generic fiber is smooth, the smooth locus of $f$ contains a nonempty open subset of $C$.

\begin{definition}
A fibered surface $f:S\to C$ is \emph{semistable} if every geometric fiber is a reduced connected nodal curve. Equivalently, each singular fiber is a reduced connected nodal curve and the geometric generic fiber is smooth. The fibration is \emph{relatively minimal} if no fiber contains a $(-1)$-curve. It is \emph{isotrivial} if its smooth part becomes a product family after a finite base change.
\end{definition}

%\begin{remark}\label{rem:generic-smoothness}
%In characteristic $p$, a morphism from a smooth surface to a smooth curve need not be generically smooth. Our definition excludes such inseparable fibrations. This allows us to use the usual identities for semistable fibrations, including $\rk f_*\omega_{S/C}=g$ and \eqref{eq:c2-fibration}.\end{remark}
%
Let $\Sigma=\{t_1,\dots,t_s\}\subset C$ be the set of singular values, and write $F_i=f^{-1}(t_i)$. We use $\omega_{S/C}=\omega_S\otimes f^*\omega_C^{-1}$ for the relative dualizing sheaf and set
\begin{equation}\label{eq:chi-f-def}
 \chi_f:=\deg f_*\omega_{S/C}
 =\chi(\OO_S)-(g-1)(b-1).
\end{equation}
For a semistable fibration, $f_*\omega_{S/C}$ is locally free of rank $g$.

\subsection{Combinatorics of singular fibers}

Write $F_i=\Gamma_{i1}+\cdots+\Gamma_{i\ell_i}$, where the $\Gamma_{ij}$ are the irreducible components and $\ell_i=\ell(F_i)$. Let $\widetilde\Gamma_{ij}$ be the normalization of $\Gamma_{ij}$. Set
$g_i=g(F_i):=\sum_{j=1}^{\ell_i}g(\widetilde\Gamma_{ij})$, and put
$$A_f:=\sum_{i=1}^s(\ell_i-1),\quad G_f:=\sum_{i=1}^s g_i.$$  Let $\delta_i$ be the number of nodes of $F_i$ and let $\delta_f=\sum_i\delta_i$.

The dual graph of a connected nodal curve gives $\delta_i=g-g_i+\ell_i-1$.  Consequently,
\begin{equation}\label{eq:delta-global}
 \delta_f=gs-G_f+A_f.
\end{equation}

\subsection{Betti numbers and the relative Albanese dimension}

Fix a prime $\ell\ne p$ and write $b_j(S):=\dim_{\mathbb Q_\ell}H^j_{\et}(S,\mathbb Q_\ell)$. For smooth proper varieties, these numbers also agree with the dimensions of rational crystalline cohomology. In particular, $b_1(S)=2\dim\Alb(S)$.

Since $f^*:H^1_{\et}(C,\mathbb Q_\ell)\hookrightarrow H^1_{\et}(S,\mathbb Q_\ell)$, we have $b_1(S)\ge2b$. Set
$$r_f:=b_1(S)-2b=2(\dim\Alb(S)-b).$$
Thus $r_f$ is the $\ell$-adic dimension of the relative Albanese part. In positive characteristic, this definition is preferable to $2(h^1(S,\OO_S)-b)$ because $\Pic^0(S)$ need not be reduced.

\subsection{Crystalline slopes, domino numbers and Hodge Witt numbers}

We briefly recall the numerical invariants of Crew and Ekedahl that replace the usual Hodge number $h^{1,1}$ in positive characteristic; see \cite{Crew1985,Ekedahl1986,Illusie1979,IllusieRaynaud1983}. The starting point is the slope spectral sequence
\begin{equation}\label{eq:slope-ss}
 E_1^{i,j}=H^j(S,W\Omega_S^i)\Longrightarrow H^{i+j}_{\cris}(S/W).
\end{equation}
After inverting $p$, the spectral sequence degenerates. Its rational part is governed by the Newton slopes of Frobenius on crystalline cohomology. Slope numbers record this rational information, while domino numbers measure the integral failure of degeneration. Hodge Witt numbers combine these two types of information and play the role of numerical Hodge numbers in positive characteristic.

Let $h^n_\lambda$ be the multiplicity of the Frobenius slope $\lambda$ in $H^n_{\cris}(S/K)$.  The slope number at the position $(i,j)$ is
\begin{equation}\label{eq:slope-number-def}
 m^{i,j}
 =\sum_{\lambda\in[i,i+1)}(i+1-\lambda)h^{i+j}_\lambda
 +\sum_{\lambda\in[i-1,i)}(\lambda-i+1)h^{i+j}_\lambda.
\end{equation}
In other words, a Newton slope is distributed linearly between the two adjacent integral slope positions. The numbers $m^{i,j}$ are nonnegative integers and satisfy the usual slope symmetries. For fixed $n$, one has $$\sum_{i+j=n}m^{i,j}=b_n.$$ In particular, a slope equal to $1$ contributes with weight one to $m^{1,1}$. This is the only part of the general definition used in Lemma~\ref{lem:m11-lower}.

The slope numbers depend only on the isocrystal $H^n_{\cris}(S/K)$. They therefore do not detect the integral torsion phenomena on the $E_1$-page of \eqref{eq:slope-ss}. These phenomena are measured by the domino numbers $T^{i,j}$ attached to the differentials
$$d:H^j(S,W\Omega_S^i)\to H^j(S,W\Omega_S^{i+1}).$$
For a smooth proper surface, Ekedahl's duality for dominoes implies that the only possibly nonzero domino number is $T^{0,2}(S)$, associated with
\begin{equation}\label{eq:domino-02}
 d:H^2(S,W\OO_S)\longrightarrow H^2(S,W\Omega_S^1).
\end{equation}
Thus $T^{0,2}(S)$ records precisely the integral defect that is invisible after tensoring with $K$.

The Hodge Witt numbers are defined by
\begin{equation}\label{eq:HW-number-def}
 h_W^{i,j}=m^{i,j}+T^{i,j}-2T^{i-1,j+1}+T^{i-2,j+2}.
\end{equation}
For a surface, the preceding vanishing of the other domino numbers gives the particularly simple identity
\begin{equation}\label{eq:HW11-surface}
 h_W^{1,1}(S)=m^{1,1}(S)-2T^{0,2}(S).
\end{equation}
Thus $m^{1,1}$ is the rational crystalline contribution. The term $2T^{0,2}$ is the integral correction that prevents it from behaving exactly like the classical $h^{1,1}$.

\begin{definition}
A smooth proper variety $X/k$ is \emph{Hodge Witt} if all groups $H^j(X,W\Omega_X^i)$ are finitely generated $W$-modules.  For a Hodge Witt variety the slope spectral sequence degenerates at $E_1$, so all domino numbers vanish.
\end{definition}

For a smooth proper surface, the only possibly nonzero domino number is $T^{0,2}(S)$. Hence $S$ is Hodge Witt if and only if $T^{0,2}(S)=0$. In this case,
$h^{1,1}_W(S)=m^{1,1}(S)$. This is the form of the Hodge Witt hypothesis used in the proof of the main theorem.

\subsection{Hodge Witt and ordinary surfaces}\label{subsec:HW-vs-ordinary}

The Hodge Witt condition is strictly weaker than ordinarity. A smooth proper variety is ordinary, in the sense of Bloch--Kato--Illusie--Raynaud, when the cohomology of the sheaves of exact differential forms vanishes in all degrees. Ordinarity implies finite generation of de Rham--Witt cohomology. Hence every ordinary variety is Hodge Witt. The converse already fails for surfaces.

K3 surfaces provide a standard example. If $X$ is a K3 surface in characteristic $p$, then $X$ is Hodge Witt if and only if its formal Brauer group has finite height; see \cite[II, \S7.2]{Illusie1979}. The surface is ordinary if and only if this height is $1$. Thus K3 surfaces of finite height $2,\ldots,10$ are Hodge Witt but nonordinary. In particular, the hypothesis of the main theorem is much weaker than ordinarity. 

There is a simple class of surfaces for which the Hodge Witt condition is automatic.

\begin{lemma}\label{lem:pg0-hodge-witt}
Let $S$ be a smooth projective surface over $k$.  If
\[
 p_g(S)=h^2(S,\mathcal O_S)=0,
\]
then $S$ is Hodge Witt.
\end{lemma}

\begin{proof}
For every $n\ge1$, the standard exact sequence
\[
0\longrightarrow W_{n-1}\mathcal O_S
\xrightarrow{\,V\,}
W_n\mathcal O_S
\longrightarrow \mathcal O_S
\longrightarrow0
\]
shows inductively that
\[
H^2(S,W_n\mathcal O_S)=0.
\]
To pass to $W\mathcal O_S$, we must also control the derived inverse limit. The restriction map $R:W_n\mathcal O_S\to W_{n-1}\mathcal O_S$
fits into
\[
0\longrightarrow V^{n-1}\mathcal O_S
\longrightarrow W_n\mathcal O_S
\xrightarrow{\,R\,}W_{n-1}\mathcal O_S
\longrightarrow0.
\]
As a sheaf of abelian groups, $V^{n-1}\mathcal O_S$ is a Frobenius twist of $\mathcal O_S$. Hence
\[
H^2(S,V^{n-1}\mathcal O_S)=0.
\]
The resulting long exact sequence therefore shows that
\[
H^1(S,W_n\mathcal O_S)\longrightarrow
H^1(S,W_{n-1}\mathcal O_S)
\]
is surjective for every $n$.  

Hence it satisfies the Mittag--Leffler
condition, and therefore
\[
\varprojlim\nolimits^1_n H^1(S,W_n\mathcal O_S)=0
\]
by \cite[Proposition~3.5.7]{Weibel1994}.

Since $W\mathcal O_S=\varprojlim_R W_n\mathcal O_S$, the standard derived
inverse-limit exact sequence gives
\[
0\longrightarrow
\varprojlim\nolimits^{1}_n H^1(S,W_n\mathcal O_S)
\longrightarrow H^2(S,W\mathcal O_S)
\longrightarrow \varprojlim_n H^2(S,W_n\mathcal O_S)
\longrightarrow0.
\]
Consequently
\[
H^2(S,W\mathcal O_S)=0.
\]
Hence the domino associated with
\[
d:H^2(S,W\mathcal O_S)\longrightarrow H^2(S,W\Omega_S^1)
\]
vanishes.  For a smooth proper surface, $T^{0,2}(S)$ is the only possibly nonzero domino number. Thus $T^{0,2}(S)=0$, and $S$ is Hodge Witt.
\end{proof}

\begin{remark}
Lemma~\ref{lem:pg0-hodge-witt} applies in particular to rational surfaces.
Thus, for surfaces with $p_g=0$, the Hodge Witt hypothesis in
Theorem~\ref{thm:hodge-witt} is automatic.
\end{remark}

The Hodge Witt condition also behaves well in families. By a theorem of Crew, if $\mathcal X\to T$ is a proper smooth family of surfaces, then the function
$$
t\longmapsto T^{0,2}(\mathcal X_t)
$$
on the geometric points of $T$ is upper semicontinuous \cite[Theorem~7]{Crew1985}. For a smooth proper surface, the Hodge Witt condition is equivalent to the vanishing of $T^{0,2}$. Hence the Hodge Witt locus in $T$ is open. If $T$ is irreducible and contains one Hodge Witt member, this locus is therefore dense and open.
%We do not, however, assert that every irreducible component of the moduli space of surfaces of general type, or of semistable fibered surfaces considered here, contains a Hodge Witt point.

For polarized K3 surfaces, the Hodge Witt condition is in fact generic in moduli. Van der Geer and Katsura stratified the moduli space by the height of the formal Brauer group. They showed that the locus of height at least $h$ has codimension $h-1$ \cite{vanDerGeerKatsura2000}. In particular, the generic K3 surface has height $1$, while the supersingular locus is the deepest stratum.

\section{Numerical and geometric estimates}\label{sec:lemmas}

\subsection{Numerical identities}

We first record the numerical identity that converts the Arakelov inequality into a lower bound for $h_W^{1,1}(S)$.

\begin{proposition}\label{prop:hW-c2}
For every smooth projective surface $S/k$,
\begin{equation}\label{eq:hW-c2}
 h_W^{1,1}(S)=c_2(S)-2\chi(\OO_S)+b_1(S).
\end{equation}
\end{proposition}

\begin{proof}
Crew's formula identifies the alternating sums of Hodge Witt numbers with those of the usual Hodge numbers \cite{Crew1985}.  Moreover, Hodge Witt numbers satisfy Hodge Witt symmetry and
\[
 \sum_{i+j=n}h_W^{i,j}=b_n(S).
\]
Thus, in degree one,
\[
 h_W^{0,1}=h_W^{1,0}=\frac{b_1(S)}2.
\]
The $i=0$ Euler characteristic formula therefore yields
\[
 1-\frac{b_1(S)}2+h_W^{0,2}=\chi(\OO_S),
\]
so
\[
 h_W^{0,2}=\chi(\OO_S)-1+\frac{b_1(S)}2.
\]
Since
\[
 b_2(S)=h_W^{0,2}+h_W^{1,1}+h_W^{2,0}
       =2h_W^{0,2}+h_W^{1,1},
\]
we get
\begin{equation}\label{h11}
 h_W^{1,1}=b_2(S)-2\chi(\OO_S)+2-b_1(S).
\end{equation}

Finally, the $\ell$-adic Euler characteristic of a smooth projective surface equals $c_2(S)$, hence
\[
 c_2(S)=2-2b_1(S)+b_2(S).
\]
Substitution to \eqref{h11} gives \eqref{eq:hW-c2}.
\end{proof}

For a semistable fibration, the Euler characteristic of the total space is
\begin{equation}\label{eq:c2-fibration}
 c_2(S)=4(g-1)(b-1)+\delta_f.
\end{equation}
Combining this with \eqref{eq:chi-f-def} yields the following key identity.

\begin{proposition}\label{prop:fundamental}
Let $f:S\to C$ be a semistable fibration.  Then
\begin{equation}\label{eq:fundamental}
 2\chi_f-g(2b-2+s)
 =2+A_f+r_f-G_f-h_W^{1,1}(S).
\end{equation}
Equivalently,
\begin{equation}\label{eq:arakelov-HW-equivalent}
 \chi_f\le\frac g2(2b-2+s)
\end{equation}
is equivalent to
\begin{equation}\label{eq:HW-target}
 h_W^{1,1}(S)\ge2+A_f+r_f-G_f.
\end{equation}
\end{proposition}

\begin{proof}
By \cref{prop:hW-c2}, \eqref{eq:c2-fibration}, and \eqref{eq:chi-f-def},
\begin{align*}
 h_W^{1,1}(S)
 &=4(g-1)(b-1)+\delta_f
   -2\bigl(\chi_f+(g-1)(b-1)\bigr)+b_1(S)\\
 &=2(g-1)(b-1)+\delta_f-2\chi_f+b_1(S).
\end{align*}
Using \eqref{eq:delta-global} and $b_1(S)=r_f+2b$, we obtain
\[
 h_W^{1,1}(S)
 =2(g-1)(b-1)+gs-G_f+A_f-2\chi_f+r_f+2b.
\]
Rearranging gives \eqref{eq:fundamental}.
\end{proof}

We next prove the two geometric estimates needed below.  Both arguments are algebraic and remain valid in positive characteristic when formulated using crystalline divisor classes and the reduced Albanese variety.

\subsection{Divisor classes and the slope number \texorpdfstring{$m^{1,1}$}{m11}}

\begin{lemma}\label{lem:m11-lower}
For a semistable fibration $f:S\to C$,
\begin{equation}\label{eq:m11-lower}
 m^{1,1}(S)\ge2+A_f.
\end{equation}
\end{lemma}

\begin{proof}
Let $F$ denote a fiber class and define
\[
 V_{\mathrm{vert}}
 :=\mathbb Q[F]
   +\sum_{i=1}^s\sum_{j=1}^{\ell_i}\mathbb Q[\Gamma_{ij}]
 \subset \NS(S)_{\mathbb Q}.
\]
For a fixed singular fiber
\[
 F_i=\Gamma_{i1}+\cdots+\Gamma_{i\ell_i},
\]
Zariski's lemma says that the intersection matrix of the $\Gamma_{ij}$ is negative semidefinite. Its radical is generated by the total fiber class $[F_i]$. All fibers are numerically equivalent, and components in distinct fibers are disjoint. Thus the classes from the $i$th singular fiber have dimension $\ell_i-1$ modulo $\mathbb Q[F]$. The contributions from distinct singular fibers are independent. Hence
\begin{equation}\label{eq:vertical-dimension}
 \dim_{\mathbb Q}V_{\mathrm{vert}}
 =1+\sum_{i=1}^s(\ell_i-1)=1+A_f,
\end{equation}
and the radical of the intersection form on $V_{\mathrm{vert}}$ is the one-dimensional space $\mathbb Q[F]$.  Note that this formulation also covers the case $s=0$.

Choose an ample divisor $H$ and put
\[
 V:=V_{\mathrm{vert}}+\mathbb Q[H]\subset\NS(S)_{\mathbb Q}.
\]
Because $H\cdot F>0$, the class $[H]$ does not belong to $V_{\mathrm{vert}}$, so
\[
 \dim_{\mathbb Q}V=2+A_f.
\]
Moreover, the intersection form restricted to $V$ is nondegenerate.  Indeed, suppose that
\[
 x=v+aH\in V,
 \qquad v\in V_{\mathrm{vert}},
\]
is orthogonal to $V$.  Intersecting with $F$ gives
\[
 0=x\cdot F=a(H\cdot F),
\]
hence $a=0$.  Then $v$ is orthogonal to $V_{\mathrm{vert}}$, so $v=cF$ by the description of the radical above.  Finally, $0=v\cdot H=c(F\cdot H)$ gives $c=0$.

Let $d=2+A_f$ and choose a $\mathbb Q$-basis $D_1,\ldots,D_d$ of $V$.  Set
\[
 W_V:=\operatorname{Span}_K\bigl\{c_1^{\cris}(D_1),\ldots,c_1^{\cris}(D_d)\bigr\}
 \subset H^2_{\cris}(S/K).
\]
We claim that these crystalline Chern classes are $K$-linearly independent.  Indeed, suppose that
\[
 \sum_{i=1}^d a_i c_1^{\cris}(D_i)=0,
 \qquad a_i\in K.
\]
Cupping with $c_1^{\cris}(D_j)$ and using \eqref{eq:crys-intersection} for every $j$ gives
\[
 \sum_{i=1}^d a_i(D_i\cdot D_j)=0.
\]
The intersection form on $V$ is nondegenerate, so the matrix $(D_i\cdot D_j)_{i,j}$ is nonsingular over $\mathbb Q$. It is therefore also nonsingular over $K$. Hence $a_1=\cdots=a_d=0$, and
\[
 \dim_K W_V=d=2+A_f.
\]

By \eqref{eq:crys-c1-frob}, $W_V$ is Frobenius-stable and Frobenius acts on it as $p$ times the scalar Frobenius.  Thus $W_V$ is isoclinic of slope $1$.  The multiplicity of slope $1$ in $H^2_{\cris}(S/K)$ is therefore at least $2+A_f$.  Since a slope equal to $1$ contributes with weight exactly $1$ to $m^{1,1}$, we conclude that
\[
 m^{1,1}(S)\ge2+A_f.
\]
\end{proof}

\begin{remark}\label{rem:no-tate}
The proof of Lemma~\ref{lem:m11-lower} uses only the explicitly constructed subspace $V\subset\NS(S)_{\mathbb Q}$.  The nondegenerate intersection matrix proves $K$-linear independence of the corresponding crystalline first Chern classes directly.  Thus no information about the cycle class map on the whole N\'eron--Severi group is needed.
\end{remark}

Combining Lemma~\ref{lem:m11-lower} with \eqref{eq:HW11-surface} gives
\begin{equation}\label{eq:hW-weak-lower}
 h_W^{1,1}(S)
 \ge2+A_f-2T^{0,2}(S).
\end{equation}

\subsection{An Albanese estimate for singular fibers}

The next statement is the positive-characteristic version of the weak Beauville-type estimate needed here. We formulate it using $b_1(S)$, or equivalently the reduced Albanese variety, rather than $h^1(S,\OO_S)$.

\begin{lemma}\label{lem:alb-fiber}
Let $F_i$ be a singular fiber of $f$.  Then
\begin{equation}\label{eq:alb-fiber}
 \dim\Alb(S)-b\le g_i.
\end{equation}
Equivalently,
\begin{equation}\label{eq:r-le-2gi}
 r_f\le2g_i.
\end{equation}
\end{lemma}

\begin{proof}
Let
\[
 \nu_i:\widetilde F_i=\coprod_{j=1}^{\ell_i}\widetilde\Gamma_{ij}
 \longrightarrow F_i\hookrightarrow S
\]
be the normalization of the reduced fiber.  After choosing base points,
$\nu_i$ induces a homomorphism of abelian varieties
\[
 \beta_i:\Jac(\widetilde F_i)
 =\prod_{j=1}^{\ell_i}\Jac(\widetilde\Gamma_{ij})
 \longrightarrow \Alb(S).
\]
Let $A_i\subset\Alb(S)$ be the image of $\beta_i$.  Since the dimension of
the image of a homomorphism of abelian varieties is at most the dimension of
its source,
\begin{equation}\label{eq:Ai-upper}
 \dim A_i\le\dim\Jac(\widetilde F_i)
 =\sum_j g(\widetilde\Gamma_{ij})=g_i.
\end{equation}

Set
\[
 Q_i:=\Alb(S)/A_i
\]
and let
\[
 \bar\alpha_i:S\longrightarrow Q_i
\]
be the composite of the Albanese morphism with the quotient map.  For every
normalized component $\widetilde\Gamma_{ij}$, the induced homomorphism
\[
 \Jac(\widetilde\Gamma_{ij})\longrightarrow Q_i
\]
is zero.  Hence the morphism $\widetilde\Gamma_{ij}\to Q_i$ is constant.
If two components meet at a node, their constants agree because both branches map to the same point of $S$. The dual graph of $F_i$ is connected. Hence all the constants agree. Thus
$\bar\alpha_i(F_i)$ is a single point.

We claim that $\dim Q_i\le b$.  Since the image of the Albanese morphism of
$S$ generates $\Alb(S)$, the image $\bar\alpha_i(S)$ generates $Q_i$.
First suppose that $\dim\bar\alpha_i(S)=2$.  Choose an ample divisor $L$ on
$Q_i$ and put $D=\bar\alpha_i^*L$.  Then $D$ is nef and big, so $D^2>0$.
Since $F_i$ is contracted,
\[
 D\cdot F_i=0,\qquad F_i^2=0.
\]
The fiber class is nonzero in $\NS(S)_{\mathbb R}$ because it has positive intersection with every ample divisor on $S$. The Hodge index theorem now gives a contradiction. Hence $\dim\bar\alpha_i(S)\le1$.

If $\bar\alpha_i(S)$ is a point, then $Q_i=0$ because this image generates
$Q_i$.  Otherwise $\bar\alpha_i(S)$ is a curve.  Again choose an ample divisor
$L$ on $Q_i$ and set $D=\bar\alpha_i^*L$.  All fibers of $f$ are numerically equivalent. Since $D\cdot F_i=0$, the restriction of $D$ to the geometric generic fiber $F_{\bar\eta}$ has degree zero.  If
$\bar\alpha_i|_{F_{\bar\eta}}$ were nonconstant, then
$(\bar\alpha_i|_{F_{\bar\eta}})^*L$ would have positive degree, a
contradiction.  Hence $\bar\alpha_i$ is constant on the geometric generic
fiber.  It therefore induces a rational map $u_i:C\dashrightarrow Q_i$ such
that $\bar\alpha_i=u_i\circ f$ over a nonempty open subset of $C$.  Since $C$
is a smooth projective curve and $Q_i$ is proper, $u_i$ extends to a morphism
$C\to Q_i$.  The two morphisms $\bar\alpha_i$ and $u_i\circ f$ agree on a dense open subset of $S$. Since $Q_i$ is separated, they agree everywhere.
Therefore
\[
 \bar\alpha_i:S\xrightarrow{f}C\xrightarrow{u_i}Q_i.
\]

Since $u_i(C)=\bar\alpha_i(S)$ generates $Q_i$, the universal property of the
Jacobian gives a surjective homomorphism
\[
 \Jac(C)\twoheadrightarrow Q_i.
\]
Consequently
\[
 \dim Q_i\le\dim\Jac(C)=b.
\]
Since $\dim\Alb(S)=\dim A_i+\dim Q_i$, we obtain
\[
 \dim A_i\ge\dim\Alb(S)-b.
\]
Combining this with \eqref{eq:Ai-upper} gives
\[
 \dim\Alb(S)-b\le g_i,
\]
which proves \eqref{eq:alb-fiber}.  Finally, using
$b_1(S)=2\dim\Alb(S)$ gives \eqref{eq:r-le-2gi}.
\end{proof}

\begin{remark}\label{rem:alb-vs-h1}

\begin{enumerate}
\item [(1)] This argument is entirely algebraic and does not use separability of any of the maps.  It is the positive-characteristic analogue of the Jacobian--Albanese argument used by Beauville and by Lu--Tan--Yu--Zuo; compare \cite{LTYZ2014}.
\item [(2)] Over $\mathbb C$ one usually writes Lemma~\ref{lem:alb-fiber} as $g(F_i)\ge q(S)-b$.  In characteristic $p$, $\Pic^0(S)$ may be nonreduced, so the characteristic-free form is $g(F_i)\ge \dim\Alb(S)-b=(b_1(S)-2b)/2$.
\end{enumerate}
\end{remark}

\begin{corollary}\label{cor:G-r}
If $s\ge2$, then
\begin{equation}\label{eq:G-ge-r}
 G_f\ge r_f.
\end{equation}
If $C=\mathbb P^1$ and $s\ge4$, then
\begin{equation}\label{eq:G-ge-2r}
 G_f\ge2r_f=2b_1(S).
\end{equation}
\end{corollary}

\begin{proof}
By Lemma~\ref{lem:alb-fiber}, each $g_i\ge r_f/2$.  Summing over the singular fibers gives
\[
 G_f\ge\frac{s}{2}r_f.
\]
The conclusions follow.
\end{proof}

\section{Proofs of the main results}\label{sec:proofs}

We first establish the refined estimate announced in \eqref{eq:refined-intro}. It is the main numerical input for Theorems~\ref{thm:hodge-witt} and~\ref{thm:P1}.

\begin{proposition}[Refined Arakelov inequality]\label{prop:correction}
For every relatively minimal semistable fibration as above,
\begin{equation}\label{eq:main-refined}
 \deg f_*\omega_{S/C}
 \le \frac g2(2b-2+s)
 +T^{0,2}(S)+\frac12\bigl(r_f-G_f\bigr).
\end{equation}
If $s\ge2$, then
\begin{equation}\label{eq:main-tail-simple}
 \deg f_*\omega_{S/C}
 \le \frac g2(2b-2+s)+T^{0,2}(S).
\end{equation}
\end{proposition}

\begin{proof}
By Proposition~\ref{prop:fundamental},
\[
 2\chi_f-g(2b-2+s)=2+A_f+r_f-G_f-h_W^{1,1}(S).
\]
Using \eqref{eq:HW11-surface} and Lemma~\ref{lem:m11-lower},
\[
 h_W^{1,1}(S)=m^{1,1}(S)-2T^{0,2}(S)
 \ge2+A_f-2T^{0,2}(S).
\]
Hence
\[
 2\chi_f-g(2b-2+s)\le r_f-G_f+2T^{0,2}(S),
\]
This gives \eqref{eq:main-refined}. If $s\ge2$, Corollary~\ref{cor:G-r} gives $G_f\ge r_f$. Hence \eqref{eq:main-tail-simple} follows.
\end{proof}

We first prove Theorem~\ref{thm:hodge-witt}.

\begin{proof}
Since $S$ is Hodge Witt, $T^{0,2}(S)=0$. Since $s\ge2$, Proposition~\ref{prop:correction} gives
\[
 \deg f_*\omega_{S/C}\le\frac g2(2b-2+s).
\]
\end{proof}

We next prove Theorem~\ref{thm:P1}.

\begin{proof}
By Nguyen's theorem \cite{Nguyen1998}, non-isotriviality implies $s\ge4$.  Here $b=0$, hence $r_f=b_1(S)$.  Since $S$ is Hodge Witt, $T^{0,2}(S)=0$.  Corollary~\ref{cor:G-r} therefore gives
\[
 G_f\ge2r_f.
\]
Applying \eqref{eq:main-refined},
\begin{align*}
 \deg f_*\omega_{S/\mathbb P^1}
 &\le \frac g2(s-2)+\frac12(r_f-G_f)\\
 &\le \frac g2(s-2)-\frac12r_f\\
 &=\frac g2(s-2)-\frac12b_1(S).
\end{align*}
This proves \eqref{eq:P1-strong}, and \eqref{eq:P1-classical} follows immediately.
\end{proof}

\begin{remark}[The case $s=4$]\label{rem:s-four}
Under the hypotheses of Theorem~\ref{thm:P1}, suppose that $s=4$.  The refined inequality \eqref{eq:P1-strong} is then necessarily an equality.  Indeed, the $i=0$ identity used in the proof of Proposition~\ref{prop:hW-c2}, together with $b=0$, gives
\[
 h_W^{0,2}(S)=\chi_f-g+\frac12b_1(S).
\]
Since $S$ is Hodge Witt, $h_W^{0,2}(S)=m^{0,2}(S)\ge0$. Hence
\[
 \chi_f=g-\frac12b_1(S)+m^{0,2}(S).
\]
On the other hand, Theorem~\ref{thm:P1} gives
\[
 \chi_f\le g-\frac12b_1(S).
\]
It follows that $m^{0,2}(S)=0$. Therefore
\[
 \chi_f=g-\frac12b_1(S).
\]
Thus $s=4$ is an equality case of the refined Arakelov inequality \eqref{eq:P1-strong}.
\end{remark}

\section{Four singular fibers in characteristic $5$}\label{sec:four-fibers}

Throughout this section $k$ is algebraically closed of characteristic $5$.  For the genus-$2$ fibration over $\mathbb P^1$ constructed below, the standard relative identities specialize to
\begin{equation}\label{eq:genus2formulas}
 K_{S/\mathbb P^1}^2=K_S^2+8,\qquad
 \chi_f=\chi(\mathcal O_S)+1,\qquad
 \delta_f=c_2(S)+4.
\end{equation}
Let $C_0$ be the negative section of the Hirzebruch surface $\mathbb F_3$ and let $\Gamma$ be a ruling fiber.  Then
\[
 C_0^2=-3,\qquad C_0\cdot\Gamma=1,\qquad \Gamma^2=0,
\]
and
\[
 K_{\mathbb F_3}=-2C_0-5\Gamma.
\]
The divisor $D=2C_0+6\Gamma$ satisfies
\[
 D^2=12,\qquad D\cdot K_{\mathbb F_3}=-10,
\]
so adjunction gives $p_a(D)=2$.

\subsection{Nguyen's pencil and its reduction modulo $5$}

The pencil used below is the genus-$2$ pencil reproduced by Nguyen in \cite[\S3.2.3]{Nguyen1995}, originating in work of M.-H.~Saito \cite{Saito1995}.  We examine its reduction in characteristic $5$ and resolve the resulting special fibers.

Let $k$ be algebraically closed of characteristic $5$.  On an affine chart of $\FF_3$ with base coordinate $x$ and fiber coordinate $w$, consider
\begin{equation}\label{eq:pencil}
 F_0+tF_1=0,
\end{equation}
where
\begin{equation}\label{eq:F0F1}
 F_0=w\bigl(w+x(x-1)\bigr),\qquad
 F_1=\bigl(x^4-(x-1)^4\bigr)w-x(x-1)^5.
\end{equation}
After homogenization, both members lie in $|2C_0+6\Gamma|$.  

Set
$
 y=2w+x(x-1)+t\bigl(x^4-(x-1)^4\bigr).
$
Then
\begin{equation}\label{eq:hyper}
 y^2=P(x,t),
\end{equation}
where
\begin{equation}\label{eq:P}
 P(x,t)=4t\,x(x-1)^5+
 \left(x(x-1)+t\bigl(x^4-(x-1)^4\bigr)\right)^2.
\end{equation}

\begin{lemma}\label{lem:discriminant}
For the affine sextic $P(x,t)$, one has
\begin{equation}\label{eq:disc-affine}
 \operatorname{Disc}_x(P)
 =4096\,t^{13}(t-1)(4t+1)(8t+1)^2.
\end{equation}
Its leading coefficient in $x$ is $4t(4t+1)$. Hence, up to a nonzero constant,
\begin{equation}\label{eq:nguyen-discriminant}
 \operatorname{Res}_x(P,P_x)
 =t^{14}(4t+1)^2(8t+1)^2(t-1).
\end{equation}
%The latter is the discriminant factor recorded by Nguyen for this example.  

Consequently, after reduction modulo $5$, the possible singular parameters are precisely
\[
 0,\qquad 1,\qquad 3,\qquad \infty.
\]
\end{lemma}

\begin{proof}
A direct discriminant calculation gives \eqref{eq:disc-affine}.  If $a_6(t)$ denotes the leading coefficient of $P$ as a polynomial in $x$, then
\[
 a_6(t)=4t(4t+1).
\]
For a sextic,
\[
 \operatorname{Res}_x(P,P_x)=(-1)^{15}a_6(t)\operatorname{Disc}_x(P),
\]
which gives \eqref{eq:nguyen-discriminant}.  

%Thus the apparent extra powers in Nguyen's displayed factor come from the leading coefficient: his factor agrees, up to a unit, with the resultant rather than with the affine polynomial discriminant divided by the leading coefficient.

Note that in characteristic zero the finite support is $0,1,-1/4,-1/8$.  In characteristic $5$,
\[
 -\frac14=1,\qquad -\frac18=3,
\]
and \eqref{eq:disc-affine} reduces, up to a nonzero scalar, to
\[
 t^{13}(t-1)^2(t-3)^2.
\]
This proves the assertion.
\end{proof}

\subsection{The fiber over $t=0$}

\begin{proposition}\label{prop:fiber0}
After resolving the base locus of the pencil \eqref{eq:pencil}, the fiber over $t=0$ is a reduced simple normal crossings divisor. It has $12$ rational irreducible components and $13$ nodes. In particular, it is semistable and
\[
 \delta(F_0)=13.
\]
The resolution requires exactly six successive blow-ups over each of the two proper base points.
\end{proposition}

\begin{proof}
The divisor $F_0=0$ is the union of
\[
 A:\ w=0,
 \qquad
 B:\ w+x(x-1)=0.
\]
The proper base points are
\[
 P_0=(0,0),\qquad P_1=(1,0).
\]
Restricting the second generator of the pencil to $A$ and $B$ gives
\begin{equation}\label{eq:restrictions}
 F_1|_A=-x(x-1)^5,
 \qquad
 F_1|_B=-x^5(x-1).
\end{equation}
Thus the local intersection multiplicities with $(A,B)$ are $(1,5)$ at $P_0$ and $(5,1)$ at $P_1$.  The restrictions in \eqref{eq:restrictions} are nonzero. Thus neither component of $F_0$ is a component of $F_1$. Hence the two generators have no common irreducible component.  In particular
\[
 I_{P_0}(F_0,F_1)=I_{P_1}(F_0,F_1)=6,
\]
and, up to a nonzero scalar,
\begin{equation}\label{eq:resultantw}
 \operatorname{Res}_w(F_0,F_1)=x^6(x-1)^6.
\end{equation}
Since $(2C_0+6\Gamma)^2=12$, these two points account for the whole base locus. 

%Since $\mathbb F_3$ is smooth and all the centres of the successive blow-ups are smooth points, the resulting surface is smooth. In the local calculations below, at each stage we divide the two generators by their common exceptional factor, so that the displayed equations define the transformed pencil. The factorizations of $F_0^{(i)}$ show inductively that every vertical exceptional component occurring in the fiber over (t=0) appears with multiplicity one. At the final infinitely near base point, the transformed generators have linearly independent linear parts; hence on the exceptional divisor of the sixth blow-up their ratio is nonconstant. Consequently the sixth exceptional divisor is horizontal, whereas the preceding five exceptional divisors are vertical.

We first resolve the chain over $P_0$.  In characteristic $5$,
\[
 x^4-(x-1)^4=-(1+x+x^2+x^3),
\]
so
\[
 F_1=-(1+x+x^2+x^3)w+x-x^6.
\]
For the first blow-up, take $w=x(1+u_1)$.  After dividing by the common exceptional factor,
\begin{align*}
 F_0^{(1)}&=x(1+u_1)(u_1+x),\\
 F_1^{(1)}&=-\bigl[x(1+x+x^2+x^4)+(1+x+x^2+x^3)u_1\bigr].
\end{align*}
The unique new base point is $x=u_1=0$.  For the second blow-up, put $u_1=x(u_2-1)$.  Then
\begin{align}
 F_0^{(2)}&=x(1-x+xu_2)u_2,\label{eq:P0F02}\\
 F_1^{(2)}&=x^3(1-x)-(1+x+x^2+x^3)u_2.\label{eq:P0F12}
\end{align}
The next base point is again the origin.  Successively setting
\[
 u_2=xu_3,\qquad u_3=xu_4,\qquad u_4=xu_5
\]
gives
\begin{align*}
 F_1^{(3)}&=x^2(1-x)-(1+x+x^2+x^3)u_3,\\
 F_1^{(4)}&=x(1-x)-(1+x+x^2+x^3)u_4,\\
 F_1^{(5)}&=(1-x)-(1+x+x^2+x^3)u_5,
\end{align*}
while $F_0^{(i)}=x\cdot(\text{unit})\cdot u_i$ for $i=2,3,4,5$.  On the fifth exceptional divisor, the last base point is
\[
 x=0,\qquad u_5=1.
\]
Writing $r=u_5-1$, the linear terms there are
\[
 F_0^{(5)}=x+\text{higher order terms},\qquad
 F_1^{(5)}=3x-r+\text{higher order terms}.
\]
They are linearly independent.  Therefore, one more blow-up removes the base point, with no seventh infinitely near point.

The calculation over $P_1$ is symmetric with the long-contact branch interchanged.  Put $r=x-1$.  Then
\[
 F_0=w(w+r+r^2),\qquad
 F_1=(1-r+r^2-r^3)w-r^5-r^6.
\]
Starting with $w=rv_1$ and then successively setting
\[
 v_1=rv_2,\qquad v_2=rv_3,\qquad v_3=rv_4,\qquad v_4=rv_5,
\]
one obtains
\begin{align*}
 F_1^{(2)}&=(1-r+r^2-r^3)v_2-r^3(1+r),\\
 F_1^{(3)}&=(1-r+r^2-r^3)v_3-r^2(1+r),\\
 F_1^{(4)}&=(1-r+r^2-r^3)v_4-r(1+r),\\
 F_1^{(5)}&=(1-r+r^2-r^3)v_5-(1+r).
\end{align*}
The final base point is $r=0$, $v_5=1$.  At that point
\[
 F_0^{(5)}=r+\text{higher order terms},\qquad
 F_1^{(5)}=(v_5-1)-2r+\text{higher order terms},
\]
so the sixth blow-up again separates the pencil completely.

Let $E_{0,1},\ldots,E_{0,5}$ and $E_{1,1},\ldots,E_{1,5}$ denote the vertical exceptional curves.  The sixth exceptional curve over each proper base point is horizontal.  

It remains only to check the intersection of $A$ and $B$ over $x=\infty$.  Set
\[
 z=x^{-1},\qquad W=wz^3.
\]
The two components become
\[
 W=0,\qquad W=-z(1-z),
\]
whose tangent lines at $z=0$ are distinct.  The second generator of the pencil is nonzero there, so this is an ordinary transverse node and not a base point.

Hence the fiber over $0$ consists of the two chains
\[
 A-E_{0,1}-E_{0,2}-E_{0,3}-E_{0,4}-E_{0,5}-B
\]
and
\[
 A-E_{1,5}-E_{1,4}-E_{1,3}-E_{1,2}-E_{1,1}-B,
\]
together with the direct transverse intersection $A\cap B$ at infinity.  Every component occurs with multiplicity one.  There are $12$ rational components and $6+6+1=13$ nodes.  
%As an independent consistency check, the dual graph of this fiber is connected and has (12) vertices and (13) edges. Hence its first Betti number is
%\[
%b_1(\Gamma_{F_0})=13-12+1=2.
%\]
%Since all irreducible components are rational and the fiber is nodal,
%\[
%p_a(F_0)
%=\sum_i g(F_{0,i})+b_1(\Gamma_{F_0})
%=0+2
%=2,
%\]
%in agreement with the genus of the general fiber.

This completes the proof.
\end{proof}

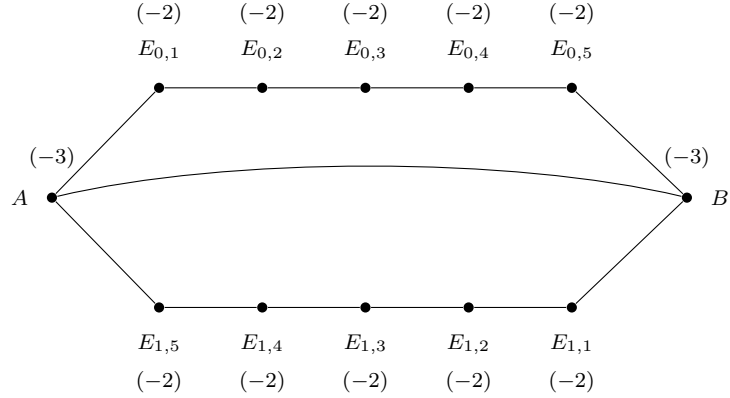
\begin{figure}[ht]
\centering
\begin{tikzpicture}[
  x=1.05cm,y=1.0cm,
  vertex/.style={circle,fill=black,inner sep=1.35pt},
  comp/.style={font=\scriptsize},
  self/.style={font=\scriptsize}
]
  % Main components A and B
  \node[vertex] (A) at (0,0) {};
  \node[vertex] (B) at (8,0) {};
  \node[comp,left=3pt of A] {$A$};
  \node[comp,right=3pt of B] {$B$};
  \node[self,above=5pt of A] {$(-3)$};
  \node[self,above=5pt of B] {$(-3)$};

  % Upper exceptional chain: E_{0,1},...,E_{0,5}
  \foreach \i/\x in {1/1.35,2/2.65,3/3.95,4/5.25,5/6.55}{
    \node[vertex] (U\i) at (\x,1.45) {};
    \node[comp,above=4pt of U\i] {$E_{0,\i}$};
    \node[self,above=18pt of U\i] {$(-2)$};
  }

  % Lower exceptional chain: E_{1,5},...,E_{1,1}
  \foreach \lbl/\x[count=\k from 1]
    in {5/1.35,4/2.65,3/3.95,2/5.25,1/6.55}{
    \node[vertex] (L\k) at (\x,-1.45) {};
    \node[comp,below=4pt of L\k] {$E_{1,\lbl}$};
    \node[self,below=18pt of L\k] {$(-2)$};
  }

  % Two chains
  \draw (A)--(U1)--(U2)--(U3)--(U4)--(U5)--(B);
  \draw (A)--(L1)--(L2)--(L3)--(L4)--(L5)--(B);

  % Curved edge representing A \cap B at infinity
  \draw (A) .. controls (2.2,0.55) and (5.8,0.55) .. (B);

\end{tikzpicture}
\caption{Dual graph of the fiber over $t=0$. 
The vertices $A$ and $B$ have self-intersection $-3$, and each vertical exceptional component
$E_{0,1},\dots,E_{0,5},E_{1,1},\dots,E_{1,5}$ has self-intersection $-2$.
The curved edge represents the transverse intersection $A\cap B$ over $x=\infty$.}
\label{fig:fiber0}
\end{figure}
%\begin{figure}[ht]
%\centering
%\begin{tikzpicture}[x=1.05cm,y=0.9cm,
%  vertex/.style={circle,fill=black,inner sep=1.25pt},
%  every node/.style={font=\scriptsize}]
%  \node[vertex] (A) at (0,0) {};
%  \node[vertex] (B) at (8,0) {};
%  \node[left=2pt of A] {$A$};
%  \node[right=2pt of B] {$B$};
%  \foreach \i in {1,...,5}{
%    \node[vertex] (U\i) at ({1.2*\i+0.2},1.2) {};
%    \node[vertex] (L\i) at ({1.2*\i+0.2},-1.2) {};
%  }
%  \draw (A)--(U1)--(U2)--(U3)--(U4)--(U5)--(B);
%  \draw (A)--(L1)--(L2)--(L3)--(L4)--(L5)--(B);
%  \draw[bend left=18] (A) to (B);
%  \node at (4,1.72) {$E_{0,1},\ldots,E_{0,5}$};
%  \node at (4,-1.72) {$E_{1,5},\ldots,E_{1,1}$};
%\end{tikzpicture}
%\caption{Dual graph of the fiber over $t=0$.  Each small black vertex represents an irreducible component.  The curved edge represents the transverse intersection $A\cap B$ over $x=\infty$.}
%\label{fig:fiber0}
%\end{figure}

\subsection{The collision fiber over $t=1$}

\begin{proposition}\label{prop:fiber1}
The fiber over $t=1$ is semistable with exactly two ordinary nodes, and the total space is smooth at both nodes.  Thus
\[
 \delta(F_1)=2.
\]
\end{proposition}

\begin{proof}
Substituting $t=1$ in \eqref{eq:P} gives
\begin{equation}\label{eq:t1affine}
 P(x,1)=-(x+2)^2(x^2-x+1).
\end{equation}
The quadratic factor has no repeated root in characteristic $5$, so the affine chart has one ordinary node at $x=3$.

The point over $x=\infty$ must also be checked.  Put
\[
 z=x^{-1},\qquad Y=yz^3,
\]
and let $Q(z,t)=z^6P(z^{-1},t)$.  Then
\begin{equation}\label{eq:t1infty}
 Q(z,1)=z^2(z-2)^2(z^2-z+1),
\end{equation}
so there is a second ordinary node at $z=0$.  Finally,
\[
 \left.\frac{\partial P}{\partial t}\right|_{(3,1)}=4\neq0,
 \qquad
 \left.\frac{\partial Q}{\partial t}\right|_{(0,1)}=1\neq0.
\]
Thus the total space is smooth at both nodes, with local form $uv=t-1$.  The proposition follows.
\end{proof}

\subsection{The fiber over $t=3$}

\begin{proposition}\label{prop:fiber3}
The fiber over $t=3$ is semistable with exactly two ordinary nodes, and the total space is smooth at both nodes.  Thus
\[
 \delta(F_3)=2.
\]
\end{proposition}

\begin{proof}
In characteristic $5$ one has
\begin{equation}\label{eq:t3factor}
 P(x,3)=(x^2-x+1)(x^2-x+2)^2.
\end{equation}
The discriminant of $x^2-x+2$ is $1-8\equiv3\pmod5$, so this quadratic has two distinct roots over $k$, each giving an ordinary node.  Moreover
\[
 \gcd\left(x^2-x+2,
 \left.\frac{\partial P}{\partial t}\right|_{t=3}\right)=1
\]
in $k[x]$.  Hence the parameter derivative is nonzero at the two nodes, so the total space is smooth there.  Therefore $\delta(F_3)=2$.
\end{proof}

\subsection{The fiber over $t=\infty$}

\begin{proposition}\label{prop:fiberinfty}
The fiber over $t=\infty$ is the union of two rational components meeting transversely in three points.  The total space is smooth at all three nodes.  Hence the fiber is semistable and
\[
 \delta(F_\infty)=3.
\]
\end{proposition}

\begin{proof}
Write $u=t^{-1}$.  Near the fiber $t=\infty$, the pencil is given on
$\mathbb F_3$ by
\[
 uF_0+F_1=0.
\]
To analyze the section at $w=\infty$, put $v=w^{-1}$.  After multiplying by
$v^2$, the equation becomes
\begin{equation}\label{eq:tinfty-original-model}
 u\bigl(1+x(x-1)v\bigr)
 +\bigl(x^4-(x-1)^4\bigr)v
 -x(x-1)^5v^2=0.
\end{equation}
In characteristic $5$,
\[
 x^4-(x-1)^4=-(1+x+x^2+x^3)
 =-(x+1)(x-2)(x+2).
\]
Thus the fiber $u=0$ is
\[
 v\Bigl(-(x+1)(x-2)(x+2)-x(x-1)^5v\Bigr)=0.
\]
It is the union of the section $v=0$ and a second rational component.  They
meet exactly at
\[
 x=-1,\qquad x=2,\qquad x=-2,
\]
and these intersections are transverse because the three roots of
$(x+1)(x-2)(x+2)$ are simple.  At each of these points $v=0$, and the partial
derivative of the left-hand side of \eqref{eq:tinfty-original-model} with
respect to $u$ equals
\[
 1+x(x-1)v=1.
\]
Hence the total space of the original pencil is smooth at all three nodes.
The base points resolved in Proposition~\ref{prop:fiber0} lie in the finite
$w$-chart and do not lie on the component $v=0$; the resolution therefore
does not alter these three local crossings.  It follows that the fiber over
$t=\infty$ is semistable and has exactly three nodes.
\end{proof}

\begin{figure}[ht]
\centering
\begin{tikzpicture}[
  x=1cm,y=1cm,
  vertex/.style={circle,fill=black,inner sep=1.35pt},
  title/.style={font=\scriptsize},
  weight/.style={font=\scriptsize}
]

%================================================
% Fibre over t = 1
%================================================
\begin{scope}[shift={(0,0)}]

  \node[title] at (0,1.45) {$t=1$};

  % One irreducible rational component
  \node[vertex] (A1) at (0,0) {};

  % Two self-nodes
  \draw (A1)
    edge[loop left,looseness=7,min distance=8mm] (A1);

  \draw (A1)
    edge[loop right,looseness=7,min distance=8mm] (A1);

\end{scope}

%================================================
% Fibre over t = 3
%================================================
\begin{scope}[shift={(4.2,0)}]

  \node[title] at (0,1.45) {$t=3$};

  % One irreducible rational component
  \node[vertex] (A3) at (0,0) {};

  % Two self-nodes
  \draw (A3)
    edge[loop left,looseness=7,min distance=8mm] (A3);

  \draw (A3)
    edge[loop right,looseness=7,min distance=8mm] (A3);

\end{scope}

%================================================
% Fibre over t = \infty
%================================================
\begin{scope}[shift={(8.6,0)}]

  \node[title] at (1.0,1.45) {$t=\infty$};

  % Two irreducible rational components
  \node[vertex] (L) at (0,0) {};
  \node[vertex] (R) at (2,0) {};

  % Their self-intersections
  \node[weight,below=5pt of L] {$(-3)$};
  \node[weight,below=5pt of R] {$(-3)$};

  % Three transverse intersection points
  \draw (L) -- (R);

  \draw (L)
    .. controls (0.7,0.65) and (1.3,0.65) ..
    (R);

  \draw (L)
    .. controls (0.7,-0.65) and (1.3,-0.65) ..
    (R);

\end{scope}

\end{tikzpicture}

\caption{Dual graphs of the stable fibers over $t=1$, $t=3$, and $t=\infty$.
Each black vertex represents an irreducible rational component, and a loop represents a self-node.
The fibers over $t=1$ and $t=3$ are irreducible rational curves with two nodes.
The fiber over $t=\infty$ consists of two rational $(-3)$-curves meeting transversely in three points.}

\label{fig:stablefibres}
\end{figure}
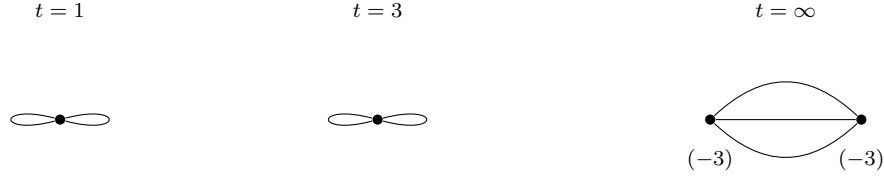

%\begin{figure}[ht]
%\centering
%\begin{tikzpicture}[vertex/.style={circle,fill=black,inner sep=1.35pt},
%  every node/.style={font=\scriptsize}]
%  % t=1
%  \node[vertex] (v1) at (0,0) {};
%  \draw (v1) edge[loop above,min distance=7mm] (v1);
%  \draw (v1) edge[loop below,min distance=7mm] (v1);
%  \node at (0,-1.05) {$t=1$};
%  % t=3
%  \node[vertex] (v3) at (4,0) {};
%  \draw (v3) edge[loop above,min distance=7mm] (v3);
%  \draw (v3) edge[loop below,min distance=7mm] (v3);
%  \node at (4,-1.05) {$t=3$};
%  % infinity
%  \node[vertex] (a) at (7.5,0) {};
%  \node[vertex] (b) at (10.5,0) {};
%  \draw (a) to[bend left=25] (b);
%  \draw (a) -- (b);
%  \draw (a) to[bend right=25] (b);
%  \node at (9,-1.05) {$t=\infty$};
%\end{tikzpicture}
%\caption{Stable dual graphs of the fibers over $1$, $3$, and $\infty$.  Each small black vertex represents an irreducible component; a loop represents a self-node.}
%\label{fig:stabletypes}
%\end{figure}

\subsection{The global fibration}

\begin{proposition}\label{prop:global}
Let $S$ be the surface obtained by the twelve blow-ups in Proposition~\ref{prop:fiber0}.  The resolved pencil defines a relatively minimal, non-isotrivial, semistable genus-$2$ fibration
\[
 f:S\longrightarrow\PP^1
\]
with exactly four singular fibers.  Moreover
\[
 K_{S/\mathbb P^1}^2=4,\qquad \deg f_*\omega_{S/\mathbb P^1}=2,\qquad \delta_f=20,
\]
and $S$ is rational and Hodge Witt.
\end{proposition}

\begin{proof}
After the twelve blow-ups of Proposition~\ref{prop:fiber0}, the pencil has no base points. It therefore defines a morphism
\[
 f:S\longrightarrow\mathbb P^1.
\]
For a finite parameter $t\notin\{0,1,3\}$, the resultant in
Lemma~\ref{lem:discriminant} and the leading coefficient of $P(x,t)$ are both
nonzero.  Thus $P(x,t)$ has six distinct roots, and the corresponding
projective hyperelliptic curve is a smooth connected curve of genus $2$,
including at the points over $x=\infty$.  Hence the geometric generic fiber of $f$ is a smooth connected genus-$2$ curve. By Stein factorization, all fibers of $f$ are connected. The surface $S$ is smooth and hence Cohen--Macaulay. The base $\mathbb P^1$ is a smooth curve, and every fiber has pure dimension one. Therefore $f$ is flat.

No singular fiber can occur away from $0,1,3,\infty$.  The resolution of the
base locus introduces vertical exceptional components only in the fiber over
$t=0$; the final exceptional curve over each of the two proper base points is
horizontal.  The fiber over $t=\infty$ is treated directly in
Proposition~\ref{prop:fiberinfty}.  Therefore the singular fibers are exactly
those over
\[
0,\quad 1,\quad 3,\quad \infty.
\]
By Propositions~\ref{prop:fiber0}--\ref{prop:fiberinfty}, all four are reduced
nodal fibers.  Thus $f$ is a semistable genus-$2$ fibration and $s=4$.

We next check relative minimality.  In the fiber over $0$, the first five exceptional curves in each base-point chain have self-intersection $-2$ after the final blow-up. The sixth exceptional curve is horizontal. The proper transforms of $A$ and $B$ have self-intersection $-3$.  The fibers over
$1$ and $3$ are irreducible, so their unique irreducible component is the
fiber class and has self-intersection $0$.

For the fiber over $\infty$, the component $v=0$ in
\eqref{eq:tinfty-original-model} is the negative section $C_0$ of
$\mathbb F_3$, hence has self-intersection $-3$.  The residual component has
class $C_0+6\Gamma$, whose self-intersection on $\mathbb F_3$ is
\[
 (C_0+6\Gamma)^2=-3+12=9.
\]
All twelve base-point blow-ups lie on this residual component. None lies on $C_0$.  Its proper transform on $S$ therefore has self-intersection
$9-12=-3$.  Hence no singular fiber contains a $(-1)$-curve, and $f$ is
relatively minimal.

The fibration is not isotrivial. A semistable isotrivial fibration of curves over $\mathbb P^1$ is smooth by uniqueness of the semistable model; see \cite[p.~41]{Nguyen1998}. The four fibers displayed above are singular. Hence $f$ is non-isotrivial.

Since $S$ is obtained from $\mathbb F_3$ by twelve point blow-ups,
\[
 K_S^2=8-12=-4,
 \qquad
 c_2(S)=4+12=16.
\]
The surface is rational, so $\chi(\mathcal O_S)=1$.  Formula
\eqref{eq:genus2formulas} gives
\[
 K_{S/\mathbb P^1}^2=4,
 \qquad
 \chi_f=2,
 \qquad
 \delta_f=20.
\]
The local node count gives the independent check
\[
 13+2+2+3=20.
\]
Since $S$ is rational, $p_g(S)=0$. Hence $S$ is Hodge Witt by Lemma~\ref{lem:pg0-hodge-witt}.

Finally,
\[
 \deg f_*\omega_{S/\mathbb P^1}=\chi_f=2
 =\frac{2}{2}(4-2)
 =\frac{g}{2}(s-2).
\]
\end{proof}

\begin{remark}
Note that in characteristic zero, we always have the strict Arakeolv inequality \cite{Tan1995} 
\[
 \deg f_*\omega_{S/\mathbb P^1}<
 \frac{g}{2}(s-2).
\]

\begin{proof}[{\bf Proof of Theorem~\ref{thm:four-fibers}}]
Combine Lemma~\ref{lem:discriminant} and Propositions~\ref{prop:fiber0}--\ref{prop:global}.
\end{proof}

\end{remark}

\section{Applications of the Arakelov inequalities}\label{sec:applications}

\subsection{A canonical-class inequality}

Canonical-class inequalities are a classical tool in the study of degenerations of curves and in function-field Diophantine geometry. Over the complex numbers, Vojta proved the logarithmic canonical-class inequality. Tan later obtained its strict form for non-isotrivial semistable fibrations:
\begin{equation}\label{eq:classical-canonical-background}
 K_{S/C}^2<(2g-2)(2b-2+s);
\end{equation}
see \cite{Vojta1987,Tan1995,Beauville2002}. Lu--Tan--Zuo developed higher-dimensional analogues \cite{LuTanZuo2017}.

The situation in positive characteristic is different. Szpiro proved that a semistable fibration with nonzero Kodaira--Spencer class satisfies
\begin{equation}\label{eq:szpiro-canonical-background}
 K_{S/C}^2<4g(g-1)(2b-2+s);
\end{equation}
see \cite[Th\'eor\`eme~3]{Szpiro1981}.  Thus this general positive-characteristic estimate has a coefficient quadratic in the fiber genus. Yuan--Zhang \cite{YuanZhang2014} combined \eqref{eq:szpiro-canonical-background} with the slope inequality and obtained
\[
 \chi_f\le g^2(2b-2+s).
\]

Under the Hodge Witt hypothesis, Theorems~\ref{thm:hodge-witt} and~\ref{thm:P1} give canonical-class inequalities with coefficients linear in $g$. Note that for $g\ge3$, the coefficient $6g$ is smaller than $4g(g-1)$. 

\begin{corollary}[Canonical-class inequality]\label{cor:canonical-class}
Under the hypotheses of Theorem~\ref{thm:hodge-witt},
\begin{equation}\label{eq:canonical-refined}
 K_{S/C}^2
 \le
 6g(2b-2+s)+6(r_f-G_f)-\delta_f .
\end{equation}
In particular,
\begin{equation}\label{eq:canonical-simple}
 K_{S/C}^2
 \le
 6g(2b-2+s)-s
 <6g(2b-2+s).
\end{equation}
If $C=\mathbb P^1$ and the hypotheses of Theorem~\ref{thm:P1} hold, then the stronger estimate
\begin{equation}\label{eq:canonical-P1-refined}
 K_{S/\mathbb P^1}^2
 \le 6g(s-2)-6b_1(S)-\delta_f
 \le 6g(s-2)-6b_1(S)-s
\end{equation}
holds.
\end{corollary}

\begin{proof}
Since $S$ is Hodge Witt, Proposition~\ref{prop:correction} gives
\[
 \chi_f\le \frac g2(2b-2+s)+\frac12(r_f-G_f).
\]
Using the relative Noether formula gives
\[
 K_{S/C}^2
 \le
 6g(2b-2+s)+6(r_f-G_f)-\delta_f.
\]
Since $s\ge2$, Corollary~\ref{cor:G-r} gives $G_f\ge r_f$. Hence
\[
 K_{S/C}^2\le6g(2b-2+s)-\delta_f.
\]
Every singular semistable fiber contains at least one node, so $\delta_f\ge s$, which proves \eqref{eq:canonical-simple}.

When $C=\mathbb P^1$, Theorem~\ref{thm:P1} and the relative Noether formula give
\[
 K_{S/\mathbb P^1}^2
 =12\chi_f-\delta_f
 \le 6g(s-2)-6b_1(S)-\delta_f.
\]
Using again $\delta_f\ge s$ proves \eqref{eq:canonical-P1-refined}.
\end{proof}

\subsection{A Szpiro-type inequality}\label{subsec:szpiro}

The Szpiro inequality is an inequality relating discriminant and conductor for elliptic curves. For a semistable elliptic fibration, a fiber of type $I_n$ has discriminant order $n$ and conductor exponent $1$. Thus the total discriminant degree counts nodes, while the conductor records the support of bad reduction. Szpiro's inequality compares these two quantities \cite{Szpiro1990}. For a semistable fibration of curves of genus $g\ge2$, the natural analogues are the total number of nodes
\[
 \delta_f=\sum_{t\in\Sigma}\delta(F_t)
\]
and the logarithmic degree $2b-2+s$.  Over $\mathbb C$, Beauville proved the strict higher-genus Szpiro inequality
\[
 \delta_f<(4g+2)(2b-2+s),
\]
by combining the slope inequality with the strict canonical-class inequality \cite{Beauville2002}.  Tan--Xu subsequently developed further optimal Szpiro-type inequalities for curves over function fields \cite{TanXu2015}.

Note that in positive characteristic, we only have the strict canonical-class inequality \eqref{eq:szpiro-canonical-background}.

Under the intrinsic Hodge Witt hypothesis, Theorem~\ref{thm:hodge-witt} recovers the same linear coefficient $4g+2$ as Beauville's higher-genus Szpiro inequality. No lifting assumption is needed, but the resulting inequality is non-strict. Theorem~\ref{thm:P1} gives an additional correction term over $\mathbb P^1$. Thus the following result is a positive-characteristic analogue of Beauville's estimate in the Hodge Witt setting.

\begin{corollary}[Szpiro-type inequality]\label{cor:szpiro}
Under the hypotheses of Theorem~\ref{thm:hodge-witt},
\begin{equation}\label{eq:szpiro}
 \delta_f\le (4g+2)(2b-2+s).
\end{equation}
If, in addition, $C=\mathbb P^1$ and the hypotheses of Theorem~\ref{thm:P1} hold, then
\begin{equation}\label{eq:szpiro-P1-refined}
 \delta_f
 \le (4g+2)(s-2)-\frac{2(2g+1)}{g}\,b_1(S).
\end{equation}
\end{corollary}

\begin{proof}
Write $K_{S/C}^2:=c_1(\omega_{S/C})^2$.  For the present semistable fibration of genus $g\ge2$, the slope inequality in arbitrary characteristic gives \cite{YuanZhang2014}
\begin{equation}\label{eq:slope-inequality}
 K_{S/C}^2\ge \frac{4(g-1)}{g}\chi_f.
\end{equation}
For a semistable fibration, the relative Noether formula is
\begin{equation}\label{eq:relative-noether}
 12\chi_f=K_{S/C}^2+\delta_f.
\end{equation}
Hence
\begin{equation}\label{55}
\delta_f
 \le \left(12-\frac{4(g-1)}g\right)\chi_f
 =\frac{4(2g+1)}g\chi_f.
\end{equation}

Applying Theorem~\ref{thm:hodge-witt} gives
\[
 \delta_f\le(4g+2)(2b-2+s).
\]
When $C=\mathbb P^1$, Theorem~\ref{thm:P1} gives instead
\[
 \chi_f\le \frac g2(s-2)-\frac12 b_1(S),
\]
Substituting this bound into \eqref{55} gives \eqref{eq:szpiro-P1-refined}.
\end{proof}

\section{Further discussion}\label{sec:further}

%\subsection{Comparison with Lu--Tan--Yu--Zuo}
%
%In  \cite{LTYZ2014}, over $\mathbb C$, Lu--Tan--Yu--Zuo prove
%\begin{equation}\label{eq:LTYZ-main}
% h^{1,1}(S)\ge 2bq_f+2+A_f,
%\end{equation}
%where $q(S)=h^1(S,\OO_S)$ and $q_f=q(S)-b$ .  The term $2+A_f$ comes from the fiber class, independent components of the singular fibers, and an ample divisor.  This part is algebraic.  In our proof it is replaced by the crystalline estimate
%\[
% m^{1,1}(S)\ge2+A_f.
%\]
%
%The extra term $2bq_f$ in \eqref{eq:LTYZ-main} comes from holomorphic one-forms on the total surface.  The complex proof uses Hodge decomposition, complex conjugation, and a Stokes argument.  These tools are not available in the same form in characteristic $p$.  We do not try to reproduce this term.  Instead, we use the Albanese estimate
%\[
% r_f\le2g(F_i)
%\]
%for every singular fiber.  The remaining difference is measured by the domino number $T^{0,2}(S)$.  The Hodge Witt condition makes this number zero.

\subsection{Why the Jost--Zuo method does not apply directly}

We recall the main idea of the Jost--Zuo proof.  Over $\mathbb C$, the variation of Hodge structure $R^1f_*\mathbb C$ gives a logarithmic Higgs bundle
\[
 (E,\theta)=\bigl(E^{1,0}\oplus E^{0,1},\theta\bigr),
 \qquad
 E^{1,0}=f_*\omega_{S/C},
 \quad
 E^{0,1}=R^1f_*\OO_S.
\]
The Higgs field is the logarithmic Kodaira--Spencer map
\[
 \theta:E^{1,0}\longrightarrow
 E^{0,1}\otimes\Omega_C^1(\log\Sigma).
\]
The polarization identifies $E^{0,1}$ with $(E^{1,0})^\vee$. A key input is that $(E,\theta)$ is semistable of degree zero. This follows from the Hermitian Yang--Mills theory used in \cite{JostZuo2002}. Let $K=\ker\theta$.  Semistability controls the degree of $K$ and the Higgs subbundle generated by the image of $\theta$.  These two estimates give
\[
 \deg E^{1,0}
 \le \frac{\rk E^{1,0}-\rk K}{2}
       \deg\Omega_C^1(\log\Sigma)
 \le \frac g2\deg\Omega_C^1(\log\Sigma).
\]
Thus the Arakelov inequality comes from semistability of the logarithmic Higgs bundle.

This argument cannot be used directly in our setting.  In positive characteristic, the logarithmic Higgs bundle attached to a semistable fibration need not be semistable.  In \cite{Ren-Xu}, we construct a generically ordinary semistable fibration with nonzero Kodaira--Spencer map whose logarithmic Higgs bundle is not semistable.  Therefore generic ordinarity and a nonzero Kodaira--Spencer map do not provide the semistability needed in the Jost--Zuo argument.  This is the main reason for using the cohomology of the total surface instead.

\subsection{The correction term}

Proposition~\ref{prop:correction} shows that the same argument works without the Hodge Witt assumption, but with one correction term.  If $s\ge2$, then
\[
 \deg f_*\omega_{S/C}
 \le \frac g2(2b-2+s)+T^{0,2}(S).
\]

\begin{example}\label{ex:supersingular-K3}
The correction term in Proposition~\ref{prop:correction} can be
nonzero even for a semistable fibration over $\mathbb P^1$.

Let $X$ be a supersingular K3 surface over $k$. Its formal Brauer group has infinite height. Moreover, $H^2_{\mathrm{cris}}(X/W)\otimes_W K$ is isoclinic of slope $1$; see, for instance, \cite{ArtinMazur1977,Ogus1979,Nygaard1980}.
In particular, $X$ is not Hodge Witt.

Choose a sufficiently ample divisor $H$ on $X$ and a general
Lefschetz pencil in $|H|$.  Such pencils exist after replacing a
very ample line bundle by a suitable power
\cite[Exp.~XVII, Thm.~2.5]{SGA7II}.  Let
$ \pi:S\longrightarrow X$
be the blow-up of the base locus of the pencil.  The pencil then
defines a semistable fibration
$ f:S\longrightarrow\mathbb P^1$
whose singular fibers have a single node.

Let $g$ be the genus of a smooth fiber.  Since $K_X=0$, adjunction
gives $H^2=2g-2.$
The base locus of the pencil therefore consists of $2g-2$ points.
Since $c_2(X)=24$, blowing up these points gives
$$
        c_2(S)=24+(2g-2)=2g+22.
$$

On the other hand, the semistable fibration formula
$$
        c_2(S)=4(g-1)(b-1)+\delta_f
$$

with $b=0$ yields
$$
        \delta_f=c_2(S)+4(g-1)=6g+18.
$$

For a general Lefschetz pencil each singular fiber has exactly one
node.  Hence
\begin{equation}\label{eq:K3-pencil-singular-fibers}
s=\delta_f=6g+18.
\end{equation}

The domino number can also be computed explicitly.  Since a
supersingular K3 surface has $b_2(X)=22$ and all slopes of
$H^2_{\mathrm{cris}}(X/W)\otimes K$ are equal to $1$, one has
$$
        m^{1,1}(X)=22.
$$

Moreover,
$$
        h_W^{1,1}(X)
        =c_2(X)-2\chi(\mathcal O_X)+b_1(X)
        =24-4=20.
$$

Using
$$
        h_W^{1,1}(X)
        =m^{1,1}(X)-2T^{0,2}(X),
$$

we obtain
$$
        T^{0,2}(X)=1.
$$

Blowing up a point adds one slope-$1$ class to $H^2_{\mathrm{cris}}$
and increases both $m^{1,1}$ and $h_W^{1,1}$ by one.  Consequently
$$
        T^{0,2}(S)=T^{0,2}(X)=1.
$$

Thus $S$ is not Hodge Witt, and the correction term occurring in
Proposition~\ref{prop:correction} is genuinely nonzero.

\end{example}

\begin{remark}
A natural question is whether the relative Frobenius geometry can control this domino number.  For a generically ordinary semistable fibration with nonzero Kodaira--Spencer map, it would be enough to prove
\[
 2T^{0,2}(S)\le G_f-r_f.
\]
This estimate would imply the classical Arakelov inequality without the Hodge Witt assumption on the total surface.
\end{remark}

\medskip
\noindent{\bf Acknowledgements}

The authors express their appreciation to Professor
Kang Zuo, Professor Rong Du and Professor Sheng-Li Tan for their help and encouragement over the years.
Both authors are supported by NSF of China. H. Sun is partially supported by  Natural Science
Foundation of Shanghai (Grant No. 23ZR1447000). W.-Y. Xu is partially supported by Science and Technology Innovation Plan of Shanghai (Grant No. 23JC1403200). %We also thank the referees sincerely for pointing out mistakes and useful detailed suggestions.

\vspace{0.2cm}
\noindent\textbf{AI declaration.}
The authors proposed the research questions and mathematical strategy of this paper. AI tools (GPT 5.6 sol) were used to assist with local calculations (including checks in the characteristic $5$ construction) and for English editing. The authors independently verified the arguments and are responsible for the final mathematical content.


\begin{thebibliography}{99}

\bibitem{Arakelov1971}
S.~J. Arakelov,
\emph{Families of algebraic curves with fixed degeneracies},
Math. USSR-Izv. \textbf{5} (1971), 1277--1302.

\bibitem{ArtinMazur1977}
M.~Artin and B.~Mazur,
\emph{Formal groups arising from algebraic varieties},
Ann. Sci. \'Ecole Norm. Sup. (4) \textbf{10} (1977), no.~1, 87--131.



\bibitem{Beauville1981}
A.~Beauville,
\emph{Le nombre minimum de fibres singuli\`eres d'un courbe stable sur $\mathbb P^1$},
Ast\'erisque \textbf{86} (1981), 97--108.

\bibitem{Beauville2002}
A.~Beauville,
\emph{The Szpiro inequality for higher genus fibrations},
in \emph{Algebraic Geometry: A Volume in Memory of Paolo Francia},
Walter de Gruyter, Berlin, 2002, 61--63.

\bibitem{ChenLuZuo2016}
K.~Chen, X.~Lu and K.~Zuo,
\emph{On the Oort conjecture for Shimura varieties of unitary and orthogonal types},
Compos. Math. \textbf{152} (2016), no.~5, 889--917.


\bibitem{Crew1985}
R.~Crew,
\emph{On torsion in the slope spectral sequence},
Compositio Math. \textbf{56} (1985), no.~1, 79--86.



\bibitem{Ekedahl1986}
T.~Ekedahl,
\emph{Diagonal complexes and $F$-gauge structures},
Travaux en Cours, vol.~18, Hermann, Paris, 1986.

\bibitem{Faltings1983}
G.~Faltings,
\emph{Arakelov's theorem for abelian varieties},
Invent. Math. \textbf{73} (1983), 337--347.


\bibitem{vanDerGeerKatsura2000}
G.~van der Geer and T.~Katsura,
\emph{On a stratification of the moduli of K3 surfaces},
J. Eur. Math. Soc. \textbf{2} (2000), no.~3, 259--290.


\bibitem{GongGu2020}
C.~Gong and Y.~Gu,
\emph{Number of singular fibres of surface fibrations over $\mathbb P^1$},
Pacific J. Math. \textbf{309} (2020), no.~2, 313--331.


\bibitem{Illusie1979}
L.~Illusie,
\emph{Complexe de de Rham--Witt et cohomologie cristalline},
Ann. Sci. \`Ecole Norm. Sup. (4) \textbf{12} (1979), no.~4, 501--661.

\bibitem{IllusieRaynaud1983}
L.~Illusie and M.~Raynaud,
\emph{Les suites spectrales associées au complexe de de Rham--Witt},
Publ. Math. IHÉS \textbf{57} (1983), 73--212.



\bibitem{JostZuo2002}
J.~Jost and K.~Zuo,
\emph{Arakelov type inequalities for Hodge bundles over algebraic varieties. Part I: Hodge bundles over algebraic curves},
J. Algebraic Geom. \textbf{11} (2002), no.~3, 535--546.


\bibitem{Liu1996}
K.-F.~Liu,
\emph{Geometric height inequalities},
Math. Res. Lett. \textbf{3} (1996), no.~5, 693--702.

\bibitem{LuShengZuo2009}
J.~Lu, M.~Sheng and K.~Zuo,
\emph{An Arakelov inequality in characteristic $p$ and upper bound of $p$-rank zero locus},
J. Number Theory \textbf{129} (2009), no.~12, 3029--3045.

\bibitem{LTYZ2014}
J.~Lu, S.-L.~Tan, F.~Yu and K.~Zuo,
\emph{A new inequality on the Hodge number $h^{1,1}$ of algebraic surfaces},
Math. Z. \textbf{276} (2014), no.~1--2, 543--555.

\bibitem{LTXZ2016}
X.~Lu, S.-L.~Tan, W.-Y.~Xu and K.~Zuo,
\emph{On the minimal number of singular fibers with non-compact Jacobians for families of curves over $\mathbb P^1$},
J. Math. Pures Appl. (9) \textbf{105} (2016), no.~5, 724--733.

\bibitem{LuZuo2017}
X.~Lu and K.~Zuo,
\emph{The Oort conjecture on Shimura curves in the Torelli locus of hyperelliptic curves},
J. Math. Pures Appl. (9) \textbf{108} (2017), no.~4, 532--552.

\bibitem{LuZuo2019}
X.~Lu and K.~Zuo,
\emph{The Oort conjecture on Shimura curves in the Torelli locus of curves},
J. Math. Pures Appl. (9) \textbf{123} (2019), 41--77.


\bibitem{LuTanZuo2017}
J.~Lu, S.-L.~Tan and K.~Zuo,
\emph{Canonical class inequality for fibred spaces},
Math. Ann. \textbf{368} (2017), no.~3--4, 1311--1332.

\bibitem{LuTanZuo2018}
X.~Lu, S.-L.~Tan and K.~Zuo,
\emph{Singular fibers and Kodaira dimensions},
Math. Ann. \textbf{370} (2018), no.~3--4, 1717--1728.


\bibitem{Nguyen1995}
K.~V. Nguyen,
\emph{On Beauville's conjecture and related topics},
J. Math. Kyoto Univ. \textbf{35} (1995), no.~2, 275--298.

\bibitem{Nguyen1998}
K.~V. Nguyen,
\emph{A remark on semi-stable fibrations over $\mathbb P^1$ in positive characteristic},
Compositio Math. \textbf{112} (1998), no.~1, 41--44.



\bibitem{Nygaard1980}
N.~O. Nygaard,
\emph{Higher de Rham--Witt complexes of supersingular $K3$ surfaces},
Compositio Math. \textbf{42} (1980), no.~2, 245--271.



\bibitem{Ogus1979}
A.~Ogus,
\emph{Supersingular $K3$ crystals},
Ast\'erisque \textbf{64} (1979), 3--86.


\bibitem{Ren-Xu}
Z.-H.~Ren and W.-Y.~Xu,
\emph{Ordinary semistable fibrations with unstable logarithmic Higgs bundles},
Preprint.

\bibitem{Saito1995}
M.-H.~Saito,
\emph{On upperbounds of Mordell--Weil ranks of higher genus fibrations},
in \emph{Algebraic Cycles and Related Topics (Kitasakado, 1994)},
World Scientific, River Edge, NJ, 1995, 37--42.


\bibitem{SGA7II}
P.~Deligne and N.~Katz (eds.),
\emph{Groupes de monodromie en g\'eom\'etrie alg\'ebrique (SGA 7 II)},
Lecture Notes in Mathematics, Vol.~340,
Springer-Verlag, Berlin--New York, 1973.

\bibitem{Szpiro1981}
L.~Szpiro,
\emph{Propri\'et\'es num\'eriques du faisceau dualisant relatif},
Ast\'erisque \textbf{86} (1981), 44--78.



\bibitem{Szpiro1990}
L.~Szpiro,
\emph{Discriminant et conducteur des courbes elliptiques},
Ast\'erisque \textbf{183} (1990), 7--18.

\bibitem{Tan1995}
S.-L.~Tan,
\emph{The minimal number of singular fibers of a semistable curve over $\mathbb P^1$},
J. Algebraic Geom. \textbf{4} (1995), no.~3, 591--596.

\bibitem{TanTuZamora2005}
S.-L.~Tan, Y.-P.~Tu and A.~G.~Zamora,
\emph{On complex surfaces with 5 or 6 semistable singular fibers over $\mathbb P^1$},
Math. Z. \textbf{249} (2005), no.~2, 427--438.


\bibitem{TanXu2015}
S.-L.~Tan and W.-Y.~Xu,
\emph{On Szpiro inequality for semistable families of curves},
J. Number Theory \textbf{151} (2015), 36--45.



\bibitem{ViehwegZuo2006}
E.~Viehweg and K.~Zuo,
\emph{Numerical bounds for semi-stable families of curves or of certain higher-dimensional manifolds},
J. Algebraic Geom. \textbf{15} (2006), no.~4, 771--791.

\bibitem{Vojta1987}
P.~Vojta,
\emph{Diophantine approximations and value distribution theory},
Lecture Notes in Mathematics, Vol.~1239,
Springer-Verlag, Berlin, 1987.

\bibitem{Weibel1994}
C.~A. Weibel,
\emph{An Introduction to Homological Algebra},
Cambridge Studies in Advanced Mathematics, vol.~38,
Cambridge University Press, Cambridge, 1994.

%\bibitem{Gros1985}
%M.~Gros,
%\emph{Classes de Chern et classes de cycles en cohomologie de
%Hodge Witt logarithmique},
%M\'em. Soc. Math. France (N.S.) \textbf{21} (1985), 1--87.
%



\bibitem{Xu2018}
W.-Y.~Xu,
\emph{On the Arakelov inequality in positive characteristic},
Math. Z. \textbf{289} (2018), no.~1--2, 109--117.

\bibitem{YuanZhang2014}
X.~Yuan and T.~Zhang,
\emph{Relative Noether inequality on fibered surfaces},
Adv. Math. \textbf{259} (2014), 89--115.

\end{thebibliography}
\end{document}